\documentclass[12pt]{amsart}
\usepackage[cp1250]{inputenc}
\usepackage[T1]{fontenc}
\usepackage{amscd}
\usepackage{amssymb}

\theoremstyle{plain}
\newtheorem{thm}{Theorem}[section]
\newtheorem{lem}[thm]{Lemma}
\newtheorem{prop}[thm]{Proposition}
\newtheorem{cor}[thm]{Corollary}

\theoremstyle{definition}
\newtheorem{dff}[thm]{Definition}
\newtheorem{rem}[thm]{Remark}
\newtheorem{exa}[thm]{Example}

\numberwithin{equation}{section} 

 \DeclareMathOperator{\sta}{star}
 
 \DeclareMathOperator{\conj}{conj}

\DeclareMathOperator{\usd}{usd}
\DeclareMathOperator{\fragd}{fragd}
\DeclareMathOperator{\Fragd}{Fragd}
\DeclareMathOperator{\Frag}{Frag}

\def\B{\mathcal{B}}
\def\W{\mathcal{W}}

\def\D{\mathcal{D}}

\def\F{\mathcal{F}}
\def\H{\mathcal{H}}

\def\P{\mathcal{P}}

\def\U{\mathcal{U}}
\def\V{\mathcal{V}}

\def\T{\mathcal{T}}

\def\phi{\varphi}
\def\p{\partial}

\def\rz{\mathbb{R}}
\def\R{\rz}

\def\wyr#1{\textit{#1}}

\def\s{\subset}

\def\t{\times}
\def\r{\rightarrow}

\def\ld{,\ldots,}

 \DeclareMathOperator{\diff}{Diff}

 \DeclareMathOperator{\cl}{cl}
 \DeclareMathOperator{\cld}{cld}
 \DeclareMathOperator{\frag}{frag}
 \DeclareMathOperator{\intt}{int}
\DeclareMathOperator{\id}{id} 
 \DeclareMathOperator{\supp}{supp}
 
\def\rr{\rho}

\keywords{ group of homeomorphisms, factorizable group, commutator
subgroup, perfectness, uniform perfectness, simplicity, uniform
simplicity, open manifold} 
\thanks{Partially supported by the Polish Ministry of Science and Higher Education and the
AGH grant n. 11.420.04}
\address{Faculty of Applied Mathematics, AGH University of Science and
 Technology, al. Mickiewicza 30, 30-059 Krak\'ow, Poland}
\email{ e-mail: imichali@wms.mat.agh.edu.pl,\quad
tomasz@uci.agh.edu.pl}
\date{June 12, 2010}

\title{On the structure of the commutator subgroup of certain homeomorphism groups}
\author{Ilona Michalik, Tomasz Rybicki}

\begin{document}

\maketitle

\begin{abstract} An important theorem of Ling states that if $G$ is any factorizable non-fixing
group of homeomorphisms of a paracompact space then its commutator
subgroup $[G,G]$ is perfect. This paper is devoted to further
studies on the  algebraic structure (e.g. uniform perfectness,
uniform simplicity) of $[G,G]$ and $[\tilde G,\tilde G]$, where
$\tilde G$ is the universal covering group of $G$. In particular,
we prove that if $G$ is bounded factorizable non-fixing group of
homeomorphisms then $[G,G]$ is uniformly perfect (Corollary 3.4).
The case of open manifolds is also investigated. Examples of
homeomorphism groups illustrating the results are given.
\end{abstract}

\section{Introduction}

Given groups $G$ and $H$, by $G\leq H$ (resp. $G\lhd H$) we denote
that $G$ is a subgroup (resp. normal subgroup) of $H$. Throughout
by $\H(X)$ we denote the group of all homeomorphism of a
topological space $X$. Let $U$ be an open subset of $X$ and let
$G$ be a subgroup of $\H(X)$. The symbol $\H_U(X)$ (resp. $G_U$)
stands for the subgroup of elements of $\H(X)$ (resp. $G$) with
support in $U$. For $g\in \H(X)$ the support of $g$, $\supp(g)$,
is the closure of $\{x\in X:\, g(x)\neq x\}$. Let $\H_c(M)$ (resp.
$G$) denotes the subgroup of $\H(M)$ (resp. $G$) of all its
compactly supported elements.

\begin{dff}\label{fac} Let $\U$ be an open cover of $X$.
A group of homeomorphisms $G$ of a space $X$ is called
\emph{$\U$-factorizable} if for every $g\in G$ there are $g_1\ld
g_r\in G$ with $g=g_1\ldots g_r$ and such that $\supp(g_i)\s U_i$,
$i=1\ld r$, for some $U_1\ld U_r\in\U$. $G$ is called
\emph{factorizable} if for every open cover $\U$ of $X$ it is
$\U$-factorizable.

Next $G$ is said to be \wyr{non-fixing} if $G(x)\neq \{ x \}$ for
every $x \in X$, where $G(x):= \{ g(x)|g \in G \}$ is the orbit of
$G$ at $x$.
\end{dff}
Given a group $G$, denote by $[f,g]=fgf^{-1}g^{-1}$ the commutator
of $f,g\in G$, and by $[G,G]$ the commutator subgroup.  Now the
theorem of Ling can be formulated as follows.

\begin{thm} \cite{li}\label{ling}
Let $X$ be a paracompact   topological space and let $G$ be a
factorizable  non-fixing group of homeomorphisms of $X$. Then the
commutator subgroup $[G,G]$ is perfect, that is
$[[G,G],[G,G]]=[G,G]$.
\end{thm}

 Recall that
a group $G$ is called \wyr{uniformly perfect} \cite{bip} if $G$ is
perfect (i.e. $G=[G,G]$) and there exists a positive integer $r$
such that any element of $G$ can be expressed as a product of at
most $r$ commutators of elements of $G$.  For $g\in [G,G]$, $g\neq
e$, the least $r$ such that $g$ is a product of $r$ commutators is
called the \wyr{commutator length} of $g$ and is denoted by
$\cl_G(g)$. By definition we put $\cl_G(e)=0$.

Throughout we adopt the following notation. Let $M$ be a
paracompact manifold of class $C^r$, where $r=0,1\ld\infty$. Then
$\D^r(M)$ (resp. $\D^r_c(M)$) denotes the group of all
$C^r$-diffeomorphisms of $M$ which can be joined with the identity
by a (resp. compactly supported) $C^r$-isotopy. For simplicity by
$C^0$-diffeomorphism we mean a homeomorphism.

Observe that in view of recent results (Burago, Ivanov and
Polterovich \cite{bip}, Tsuboi \cite{Tsu2}) the diffeomorphism
groups $\D^{\infty}_c(M)$ are uniformly perfect for most types of
manifolds $M$, though some open problems are left.

Our first aim is to prove the following generalization of Theorem
1.2.

\begin{thm}
Let $X$ be a paracompact   topological space and let $G$ be a
factorizable  non-fixing group of homeomorphisms of $X$. Assume
that $\cl_G$  is bounded on $[G,G]$ and that $G$ is bounded with
respect to all fragmentation norms $\frag^{\U}$ (c.f. section 2),
where $\U$ runs over all open covers  of $X$. Then the commutator
subgroup $[G,G]$ is uniformly perfect.
\end{thm}

The proof of Theorem 1.3 and further results concerning the
uniform perfectness of $[G,G]$ will be given in section 3.

Ling's theorem (Theorem 1.2) constitutes an essential amelioration
of the simplicity Epstein theorem \cite{ep} at least in two
aspects. First, contrary to \cite{ep}, it provides an algebraic
information on nontransitive homeomorphism groups.  Second, it
enables to strengthen the theorem of Epstein itself. We will
recall Epstein's theorem and Ling's improvement of it in section
4. Also in section 4 we formulate conditions which ensure the
uniform simplicity of $[G,G]$ (Theorem 4.3).

As usual $\tilde G$ stands for the universal covering group of
$G$. In section 5 we will prove the following

\begin{thm}
Suppose that  $G\leq\H(X)$ is isotopically factorizable (Def. 5.2)
and that $G_0$, the identity component of $G$, is non-fixing. Then
the commutator group $[\tilde G,\tilde G]$ is perfect.
\end{thm}

In section 6 we will consider the case of a noncompact manifold
$M$ such that  $M$ is the interior of a compact manifold $\bar M$,
and groups of homeomorphisms on $M$ with no restriction on
support. Consequently such groups are not factorizable in the
usual way but only in a wider sense (Def. 6.1).  It is surprising
that for a large class of homeomorphism or diffeomorphism groups
of an open manifold the assertions of Theorems 1.2 and 1.3 still
hold (see Theorems 6.9 and 6.10).

 In the
final section we will present some examples and open problems
which are of interest in the context of the  above results.

{\bf Acknowledgments.} A correspondence with Paul Schweitzer  and
his recent paper \cite{Sch} were helpful when we were preparing
section 6. We would like to thank him very much for his kind help.

\section{Conjugation-invariant norms}

The notion of the conjugation-invariant norm is a basic tool in
studies on the structure of groups.  Let $G$ be a group. A
\wyr{conjugation-invariant norm} (or \emph{norm} for short) on $G$
is a function $\nu:G\r[0,\infty)$ which satisfies the following
conditions. For any $g,h\in G$ \begin{enumerate} \item $\nu(g)>0$
if and only if $g\neq e$; \item $\nu(g^{-1})=\nu(g)$; \item
$\nu(gh)\leq\nu(g)+\nu(h)$; \item $\nu(hgh^{-1})=\nu(g)$.
\end{enumerate}
Recall that a group is called \emph{ bounded} if it is bounded
with respect to any bi-invariant metric. It is easily seen that
$G$ is bounded if and only if any conjugation-invariant norm on
$G$ is bounded.

Observe that the commutator length $\cl_G$ is a
conjugation-invariant norm on $[G,G]$. In particular, if $G$ is a
perfect group then $\cl_G$  is a conjugation-invariant norm on
$G$. For any perfect group $G$ denote by $\cld_G$ the commutator
length diameter of $G$, i.e. $\cld_G:=\sup_{g\in G}\cl_G(g)$. Then
$G$ is uniformly perfect iff $\cld_G<\infty$.

Assume now that $G\leq\H(X)$ is $\U$-factorizable (Def.1.1), and
that $\U$ is a $G$-invariant open cover of $X$. The latter means
that $g(U)\in\U$ for all $g\in G$ and $U\in\U$. Then we may
introduce the following conjugation-invariant norm $\frag^{\U}$ on
$G$. Namely, for $g\in G$, $g\neq\id$, we define $\frag^{\U}(g)$
to be the least integer $\rr>0$ such that $g=g_1\ldots g_{\rr}$
with $\supp(g_i)\s U_i$ for some $U_i\in\U$, where $i=1\ld \rr$.
By definition $\frag^{\U}(\id)=0$.

Define $\fragd^{\U}_G:=\sup_{g\in G}\frag^{\U}(g)$, the diameter
of $G$ in $\frag^{\U}$. Consequently, $\frag^{\U}$ is bounded iff
$\fragd^{\U}_G<\infty$.

Observe that $\frag^{\{X\}}$ is the trivial norm on $G$, i.e.
equal to 1 for all $g\in G\setminus\{\id\}$. Observe as well that
$\frag^{\V}\geq\frag^{\U}$ provided $\V$ is finer than $\U$.

The significance of $\frag^{\U}$ consists in the following version
of Proposition 1.15 in \cite{bip}.
\begin{prop}
Let $M$ be a $C^r$-manifold, $r=0,1\ld\infty$. Then $\D^r_c(M)$ is
bounded if and only if $\D^r_c(M)$ is bounded with respect to
$\frag^{\U}$, where $\U$ is some cover by embedded open balls.
\end{prop}

Indeed, it is a consequence of Theorem 1.18 in \cite{bip} stating
that for a portable manifold $M$ the group $\D^r_c(M)$ is bounded,
and the fact that $\R^n$ is portable.

\section{Uniform perfectness of $[G,G]$}

In Theorems 3.5 and 3.8 below we also need  stronger notions than
that of non-fixing group (Def. 1.1).

\begin{dff}  Let $\U$ be an open cover of
$X$, $G\leq\H(X)$ and let $r\in\mathbb N$. \begin{enumerate} \item
$G$ is called \emph{$r$-non-fixing} if for any $x\in X$ there are
$f_1\ld f_r,g_1\ld g_r\in G$ (possibly $=\id$) such that
$([f_r,g_r]\ldots[f_1,g_1])(x)\neq x$. \item $G$ is said to be
\emph{$\U$-moving} if for every $U\in\U$ then there is $g\in G$
such that $g(U)\cap U=\emptyset$. \item $G$ is said to be
\emph{$r$-$\U$-moving} if for any $U\in \U$  there are $2r$
elements of $G$ (possibly $=\id$), say $f_1\ld f_r, g_1\ld g_r$,
such that the sets $U$ and $([f_r,g_r]\ldots[f_1,g_1])(U)$ are
disjoint. \item $G$ is said to be \emph{strongly $\U$-moving} if
for every $U, V\in\U$ there is $g\in G$ such that $g(U)\cap (U\cup
V)=\emptyset$. \item $G$ is called \emph{locally moving} if for
any open set $U\s X$ and $x\in U$ there is $g\in G_U$ such that
$g(x)\neq x$.
\end{enumerate}
\end{dff}

Of course, if $G$ is either $r$-non-fixing, or $\U$-moving, or
locally moving then it is non-fixing. Likewise, if $G$ is
$r$-$\U$-moving then it is $s$-$\U$-moving for $r<s$ and
$\U$-moving. Notice that if $\V$ is finer than $\U$ and $G$ is
(resp. strongly) $\U$-moving then $G$ is (resp. strongly)
$\V$-moving.

\begin{prop} Let $X$ be paracompact and let
 $G\leq\H(X)$.\begin{enumerate} \item
  If $G$ is non-fixing  and factorizable (Def. 1.1)  then
$G$ is locally moving. \item If $G$ is locally moving then so is
$[G,G]$. \item If $G$ is non-fixing  and factorizable then $[G,G]$
is $1$-non-fixing (Def. 3.1(1)).
\end{enumerate}
\end{prop}

\begin{proof} (1) Let $x\in U$ and $g\in G$ such that $g(x)=y\neq x$.
Choose $\U=\{U_1, U_2\}$, where $x\in U_1\setminus U_2$, $y\in
U_2\setminus U_1$, $U_1\s U$ and $X=U_1\cup U_2$. By assumption
 we may write $g=g_r\ldots g_1$,
where all $g_i$ are supported in elements of $\U$. Let
$s:=\min\{i\in\{1\ld r\}:\;  \supp(g_i)\s U_1 \text{\; and\;}
g_i(x)\neq x\}$. Then $g_s\in G_U$ satisfies $g_s(x)\neq x$.

(2) Let $x\in U$. There is $g\in G_U$ with $g(x)\neq x$. Take an
open $V$ such that $x\in V\s U$ and $g(x)\not\in V$. Choose $f\in
G_V$ with $f(x)\neq x$. It follows that $f(g(x))=g(x)\neq g(f(x))$
and, therefore, $[f,g](x)\neq x$. (3) follows from (1) and the
proof of (2).
\end{proof}

The following property of paracompact spaces is well-known.
\begin{lem}\label{cover}
If $X$ is a paracompact space and $\U$ is an open cover of $X$,
then there exists an open cover $\V$ star finer than $\U$, that is
 for all $V\in \V$  there
is $U\in\U$ such that $\sta^{\V}(V)\s U$. Here
$\sta^{\V}(V):=\bigcup\{V'\in\V:\; V'\cap V\neq\emptyset\}$. In
particular, for all $V_1, V_2\in \V$ with $V_1\cap
V_2\neq\emptyset$ there is $U\in\U$ such that $V_1\cup V_2\subset
U$.
\end{lem}
If $\V$ and $\U$ are as in Lemma 3.3 then we will write
$\V\prec\U$.

For an open cover $\U$ let $\U^G:=\{g(U):\; g\in G \text{\; and
\;} U\in\U\}$.

\bigskip \emph{Proof of Theorem 1.3.} In view of Proposition 3.2
and the assumption, for any $x\in X$ there is $f,g\in[G,G]$ such
that $[f,g](x)\neq x$. It follows the existence of an open cover
$\U$ such that for any $U\in\U$ there are $f,g\in[G,G]$ such that
$[f,g](U)\cap U=\emptyset$. Hence we have also that for any
$U\in\U^G$ there is $f,g\in[G,G]$ such that $[f,g](U)\cap
U=\emptyset$. In fact, if $N\lhd G$ and $U\in\U$ such that
$n(U)\cap U=\emptyset$ for some $n\in N$, then for $g\in G$ we get
$(\bar ng)(U)\cap g(U)=\emptyset$, where $\bar n=gng^{-1}\in N$.

Due to Lemma 3.3  we can find $\V$ such that $\V\prec\U$. We
denote
\begin{equation*}G^{\U}=\prod\limits_{U\in \U^G}[G_U,G_U].\end{equation*}
Assume   that $G$ is $\V$-factorizable and $\fragd^{\V}_G= \rr$.
 First we show that
$[G,G]\s G^{\U}$ and that any $[g_1,g_2]\in[G,G]$ can be expressed
as a product of at most $\rr^2$  elements of $G^{\U}$ of the form
$[h_1,h_2]$, where $h_1,h_2\in G_U$ for some $U$. In fact, it is
an immediate consequence of the following commutator formulae for
all $f,g,h\in G$
\begin{equation}
[fg,h]=f[g,h]f^{-1}[f,h],\quad [f,gh]=[f,g]g[f,h]g^{-1},
\end{equation}
and the fact that  $\V\prec\U$. Now if $\cld_G=d$, then every
element of $[G,G]$ is a product of at most $d\rr^2$ elements of
$G^{\U}$ of the form $[h_1,h_2]$, where $h_1,h_2\in G_U$ for some
$U$.

 Next, fix arbitrarily $U\in \U$. We have to show that for
every $f,g\in G_U$ the bracket $[f,g]$ can be represented as a
product of four commutators of   elements of $[G,G]$.
By assumption on $\U^G$, there are $h_1,h_2\in [G,G]$ such that
$h(U)\cap U=\emptyset$, where $h=[h_1,h_2]$. It follows that
$[hfh^{-1}, g]=\id$. Therefore, $[[h,f],g]=[f,g]$. Observe that
indeed $[[h,f],g]$ is a product of four commutators of   elements
of $[G,G]$. Thus any element of $[G,G]$ is a product of at most
$4d\rr^2$ commutators of elements of $[G,G]$. \quad $\square$

\medskip

\begin{cor}
Let $X$ be a paracompact    space and let $G\leq\H(X)$ be a
bounded, factorizable and non-fixing group. Then the commutator
subgroup $[G,G]$ is uniformly perfect.
\end{cor}
\begin{proof}
The only thing we need is that $\cl_G$ should be bounded (on
$[G,G]$), and this fact is a consequence of Proposition 1.4 in
\cite{bip}.
\end{proof}

A more refined version of Theorem 1.3 is the following

\begin{thm}
Let $X$ be a paracompact   topological space, let $G\leq\H(X)$
with $\cl_G$ bounded (as the norm on $[G,G]$) and let $\U$ be a
$G$-invariant open cover of $X$ such that
\begin{enumerate} \item $G$ is strongly $\U$-moving (Def. 3.1(4)), and \item  there is an open
cover $\V$  satisfying $\V\prec\U$   such that $G$ is
$\V$-factorizable and $G$ is bounded with respect to the
fragmentation norm $\frag^{\V}$.\end{enumerate} Then the
commutator subgroup $[G,G]$ is uniformly perfect. Furthermore, if
$\fragd^{\V}_G= \rr$ and $\cld_G=d$ then $\cld_{[G,G]}\leq
d\rr^2$.
\end{thm}
\begin{proof}
 Let $\U$ and $\V$ satisfy the
assumption.  We denote
$$G^{\U}=\prod\limits_{U\in \U}[G_U,G_U].$$
As in the proof of 1.3, first we show, due to (3.1) and
$G$-invariance of $\U$, that $[G,G]\s G^{\U}$ and that any
$[f,g]\in[G,G]$ can be written as a product of at most $\rr^2$
elements of $G^{\U}$ of the form $[h_1,h_2]$, where $h_1,h_2\in
G_U$ for some $U$. This implies that every element of $[G,G]$ is a
product of at most $d\rr^2$  elements of $G^{\U}$ of the form
$[h_1,h_2]$, where $h_1,h_2\in G_U$ for some $U$.

For $U\in \U$ we will show that for every $f,g\in G_U$ the bracket
$[f,g]$ is a commutator of  two elements of $[G,G]$.
By assumption and Def. 3.2(4), there is $h\in G$ such that
$h(U)\cap U=\emptyset$. It follows that $[hfh^{-1}, g]=\id$. Next,
for $U, h(U)\in\U$ there is $k\in G$ such that $k(U)\cap(U\cup
h(U))=\emptyset$. Consequently,  $[f, kgk^{-1}]=\id$ and
$[hfh^{-1}, kgk^{-1}]=\id$. Therefore, in view of (3.1),
$[f,g]=[[f,h],[g,k]]$, that is $[f,g]$ is a commutator of elements
of $[G,G]$. Thus $[G,G]$ is uniformly perfect and
$\cld_{[G,G]}\leq d\rr^2$, as required.
\end{proof}

From the proof of  Theorem 3.5 we get
\begin{cor}
If $\U$ is a $G$-invariant open cover of $X$ such that  $G$ is
strongly $\U$-moving  and  $\V$-factorizable for some open cover
$\V$ satisfying $\V\prec\U$   then $[G,G]$ is perfect.
\end{cor}

\begin{prop}
(1) Let  $G$ be $\U$-moving. Assume that $\V$ is a $G$-invariant
open cover
 such that $\V\prec\U$, $G$ is $\V$-factorizable and $\fragd^{\V}_G=\rr$.
Then $G$ is $\rr$-$\V$-moving.


(2) Let $\U$, $\V$, $\W$ and $\T$ be such that
$\T\prec\W\prec\V\prec\U$, and  $\V$, $\W$ and $\T$ are
$G$-invariant. If $G$ is $\U$-moving and $\T$-factorizable with
$\fragd^{\T}_G=\rr$, then $[G,G]$ is $\rr^2$-$\W$-moving.
\end{prop}
\begin{proof} (1) Suppose that $\V\prec\U$ and
let  $V\in\V$. Then there is $g\in G$ such that $g(V)\cap
V=\emptyset$. By assumption there exist $V_1\ld V_{\rr}\in\V$ and
$g_1\ld g_{\rr}\in G$ such that $g=g_r\ldots g_1$ and
$\supp(g_i)\s V_i$, $i=1\ld \rr$ (possibly $g_i=\id$).

Let us consider two cases: $(a)$ $g_1(V)\cap V=\emptyset$ and
$(b)$ $g_1(V)\cap V\neq\emptyset$. In case $(a)$ we have $
g_1(V)\cup V\s\supp(g_1)\s U\in\U$. Choose $f_1\in G$ such that
$f_1(U)\cap U=\emptyset$. Then $[g_1,f_1](V)=g_1(V)$ and we are
done. In case $(b)$ $V\cup g_1(V)\s U_1\in\U$ such that
$f_1(U_1)\cap U_1=\emptyset$ for some $f_1\in G$. Again
$[g_1,f_1](V)=g_1(V)$. Now we continue as before. In case
$(g_2g_1)(V)\cap g_1(V)=\emptyset$ we get $V\cap\bar
g_2(V)=\emptyset$, where $\bar g_2=g_1^{-1}g_2^{-1}g_1$, and we
are done as in $(a)$. Otherwise, $(g_2g_1)(V)\cup g_1(V)\s
U_2\in\U$ such that $f_2(U_2)\cap U_2=\emptyset$ for some $f_2\in
G$. Therefore, $[g_2,f_2](g_1(V))=(g_2g_1)(V)$. Proceeding by
induction we get
$$ ([g_{\rr},f_{\rr}]\ldots [g_1,f_1])(V)=(g_{\rr}\ldots g_1)(V)=g(V),$$
and the claim follows.

(2) It follows from the hypotheses that $G$ is $\V$-factorizable
and $\fragd^{\V}_G\leq \rr$. Moreover, as in the proof of Theorem
1.3 we get that $[G,G]$ is $\W$-factorizable and
$\fragd^{\W}_{[G,G]}\leq \rr^2$. Hence by (1) $G$ is
$\rr$-$\V$-moving. In particular $[G,G]$ is $\V$-moving. Then
again (1) implies that $[G,G]$ is $\rr^2$-$\W$-moving.
\end{proof}

In the following  version of Theorem 1.3 we avoid the assumption
that $G$  is strongly $\U$-moving.
\begin{thm}
Let $X$ be a paracompact   topological space, let $G\leq\H(X)$
with $\cl_G$ bounded, and let $\U$ be an open cover of $X$ such
that
\begin{enumerate} \item $G$ is  $\U$-moving, and \item  there are  $G$-invariant open
covers $\V$, $\W$, and $\T$ fulfilling the relation
$\T\prec\W\prec\V\prec\U$,
 and such that $G$ is $\T$-factorizable and it is
bounded with respect to $\frag^{\T}$.\end{enumerate} Then  $[G,G]$
is uniformly perfect and $\cld_{[G,G]}\leq 4d\rr^4$ provided
$\fragd^{\T}_G= \rr$ and $\cld_G=d$.
\end{thm}
\begin{proof}
Let $\fragd_G^{\T}=\rr$. Then a fortiori $\fragd_G^{\W}\leq \rr$.
In view of Proposition 3.7, $[G,G]$ is $\rr^2$-$\W$-moving.

Let $[f,g]\in[G,G]$. By applying for $\T\prec\W$ the same
reasoning as in the proof of Theorem 1.3 for $\V\prec\U$, $[f,g]$
can be written as a product of at most $\rr^2$ elements from
$G^{\W}=\prod_{W\in\W}[G_W,G_W]$ of the form $[h_1,h_2]$, where
$h_1,h_2\in G_W$ for some $W\in \W$. Consequently, every element
of $[G,G]$ can be expressed as a product of at most $d\rr^2$
elements of $G^{\W}$ of the form $[h_1,h_2]$, where $h_1,h_2\in
G_W$ for some $W\in\W$.

Now take arbitrarily $W\in\W$ and $f,g\in G_W$. Since $[G,G]$ is
$\rr^2$-$\W$-moving, there are $h_1\ld h_{\rr^2},h'_1\ld
h'_{\rr^2}\in [G,G]$ such that for $h=[h_1,h'_1]\ldots
[h_{\rr^2},h'_{\rr^2}]$ we have $h(W)\cap W=\emptyset$ and,
consequently, $[[h,f],g]=[f,g]$. It is easily seen that
$[[h,f],g]$ is a product of $4\rr^2$ commutators of elements of
$[G,G]$. Thus any element of $[G,G]$ is a product of at most
$4d\rr^4$ commutators of elements of $[G,G]$.

\end{proof}
As a consequence of the above proof we have
\begin{cor}
If $G$ is  $\U$-moving and  $\T$-factorizable for some
$G$-invariant open covers $\V$, $\W$, and $\T$ such that
$\T\prec\W\prec\V\prec\U$, then $[G,G]$ is  perfect.
\end{cor}

\section{Simplicity and uniform simplicity of $[G,G]$}

Let us recall Epstein's theorem.
\begin{thm}\cite{ep}\label{eps}
Let $X$ be a paracompact space, let $G$ be a group
of~homeomorphisms of $X$ and let $\,\mathcal{B}$ be a basis of
open
sets of $X$ satisfying the following axioms:\\
\noindent{Axiom 1.} If $U\in \mathcal{B}$ and $g\in G$,
then $g(U)\in \mathcal{B}$.\\
\noindent{Axiom 2.} $G$ acts transitively on
$\mathcal{B}$ (i.e. $\forall\, U,V \in \B \; \exists\, g \in G : g(U)=V$).\\
\noindent{Axiom 3.} Let $g\in G$, $U\in \mathcal{B}$ and let
$\mathcal{U}\subset \mathcal{B}$ be a cover of $X$. Then there
exist an integer $n$, elements $g_1,\dots ,g_n\in G$ and
$V_1,\dots ,V_n\in \mathcal{U}$ such that $g=g_ng_{n-1}\dots g_1$,
$\supp (g_i)\subset V_i$ and
$$\supp (g_i)\cup (g_{i-1}\dots g_1(\overline{U}))\neq X\; \text{for}\; 1\leqslant i\leqslant n.$$
Then $[G,G]$, the commutator subgroup of $G$, is simple.
\end{thm}

 It is worth noting that Theorem 4.1 was an indispensable
ingredient in the proofs of celebrated simplicity theorems on
diffeomorphism groups and their generalizations (c.f. \cite{Thu},
\cite{Mat}, \cite{ban1}, \cite{ban2}, \cite{ha-ry}, \cite{ry1}).

We say that $G\leq\H(X)$ acts \emph{transitively inclusively}
(c.f. \cite{li}) on a topological basis $\B$ if for all $U,V\in\B$
there is $g\in G$ such that $g(U)\s V$. It is not difficult to
derive from Theorem 1.2 the following amelioration of Theorem 4.1,
see \cite{li}.

\begin{thm}\cite{li}
Let $X$ be a paracompact space, let $G\leq \H(X)$  and let
$\,\mathcal{B}$ be a basis of open
sets of $X$ satisfying the following axioms:\\
\noindent{Axiom 1.} $G$ acts transitively
inclusively  on  $\B$.\\
\noindent{Axiom 2.} $G$ is $\U$-factorizable (Def. 1.1) for all
covers $\U\s\B$.\\
 Then $[G,G]$
is a simple group.
\end{thm}

Now we wish  to provide conditions ensuring that the commutator
group of a homeomorphism group is uniformly simple. Recall that a
group $G$ is called \emph{uniformly simple} if there is $d>0$ such
that for all $f,g\in G$ with $f\neq e$ we have
$g=h_1fh_1^{-1}\ldots h_sfh_s^{-1}$, where $s\leq d$ and $h_1\ld
h_s\in G$. Given a uniformly simple group $G$, denote by $\usd_G$
the least $d$ as above.

Note that recently Tsuboi \cite{Tsu3} showed that
$\D^{\infty}_c(M)$ is uniformly simple for many types of manifolds
$M$. However, for some types of $M$ the problem is still unsolved.

\begin{thm} Let $\B$ be a topological basis of $X$. Suppose that $G\leq\H(X)$
satisfies the following conditions:\begin{enumerate} \item $\cl_G$
is bounded; \item $G$ acts transitively inclusively on $\B$; \item
there is an open cover $\U\prec\B$ such that $G$ is
$\U$-factorizable and $G$ is bounded w.r.t. the fragmentation norm
$\frag_G^{\U}$.\end{enumerate} Then the group $[G,G]$ is uniformly
simple. Moreover, if $\cld_G=d$ and $\fragd_G^{\U}=\rr$ then
$\usd_G\leq 4d\rr^2$.
\end{thm}

\begin{proof}
 In view of Theorem 4.2, $[G,G]$ is
simple. Let $f,g\in [G,G]$ such that $f\neq e$. There is $x\in X$
with $f(x)\neq x$ and $B\in\B$ satisfying $f(B)\cap B=\emptyset$.

 First we assume that  $g=[g_1,g_2]\in[G,G]$. Then, if $\fragd_G^{\U}=\rr$ then $g$
 can be expressed
as a product of at most $\rr^2$  elements of $G^{\B}=\prod_{U\in
\B^G}[G_U,G_U]$ of the form $[h_1,h_2]$, where $h_1,h_2\in G_U$
for some $U\in\B^G$. Here $\B^G=\{g(U)|\; g\in G,\; U\in\B\}$. In
fact, we repeat the use of (3.1) as in the proof of Theorem 3.1.
Now if $\cld_G=d$, then every $g\in[G,G]$ is a product of at most
$d\rr^2$ elements of $G^{\B}$ of the form $[h_1,h_2]$, where
$h_1,h_2\in G_U$ for some $U\in\B^G$.

Since $G$ acts transitively inclusively  on $\B$ (and,
consequently, on $\B^G$), any $[h_1,h_2]$ as above is conjugate to
$[k_1,k_2]$ with $k_1,k_2\in G_B$. Then $[k_1,k_2]=[[f,k_1],k_2]$.
Hence $[k_1,k_2]$ is a product of four conjugates of $f$ and
$f^{-1}$.  It follows that $g$ is a product of at most $4d\rr^2$
conjugates of $f$ and $f^{-1}$, as claimed.

\end{proof}
\begin{cor}
If $G\leq\H(X)$ is factorizable and bounded, and $G$ acts
transitively inclusively  on some basis $\B$ of $X$, then $[G,G]$
is uniformly simple.
\end{cor}
In fact, in view of Proposition 1.4 \cite{bip} $[G,G]$ is then
bounded in $\cl_G$, and the remaining hypotheses of Theorem 4.3
are fulfilled too.

\section{Perfectness of $[\tilde G,\tilde G]$}

Let $G$ be  a topological group. By $\mathcal{P}G$ we will denote
the totality of paths  (or isotopies) $\gamma:I\r G$ with
$\gamma(0)=e$ (where $I=[0,1]$). Then $\mathcal{P}G$ endowed with
the pointwise multiplication is a topological group. Next, $\tilde
G$ will stand for the universal covering group of $G$, that is
$\tilde G=PG/_{ \sim}$, where $\sim$ denotes the relation of the
homotopy rel. endpoints.

We introduce the following two operations on the space of paths
$\P G$. Let $\mathcal{P}^{\star}G=\{ \gamma \in \mathcal{P}G:
\gamma (t)=e \quad \textrm{for} \quad t \in [0,\frac{1}{2}] \}$.
For all $\gamma \in \P G$ we define $\gamma^{\star}$ as follows:

\begin{equation}\nonumber
 \gamma^{\star}(t)=
\left\{
\begin{array}{lcl}
 e& for & t \in [0,\frac{1}{2}]\\
\gamma(2t-1)& for& t \in [\frac{1}{2},1]
\end{array}
\right.
\end{equation}
Then $\gamma^{\star}\in\P^{\star}G$ and  the subgroup $P^{\star}G$
is the image of $\P G$ by the mapping $\star:\gamma\mapsto
\gamma^{\star}$. The elements of $\P^{\star}G$ are said to be
\wyr{special} paths in $G$. Clearly, the group of special paths is
preserved by conjugations, i.e. for each $g\in\P G$ we have
$\conj_g(\P^{\star}G)\s\P^{\star}G$ for every $g\in\P G$, where
$\conj_g(h)=ghg^{-1}$, $h\in\P G$.

Next, let $\mathcal{P}^{\square}G=\{ \gamma \in \mathcal{P}G:
\gamma (t)=\gamma(1) \quad \textrm{for} \quad t \in [\frac{1}{2},
1] \}$. For all $\gamma \in \P G$ we define $\gamma^{\square}$ by:

\begin{equation}\nonumber
 \gamma^{\square}(t)=
\left\{
\begin{array}{lcl}
 \gamma(2t)& for & t \in [0,\frac{1}{2}]\\
\gamma(1)& for& t \in [\frac{1}{2},1]
\end{array}
\right.
\end{equation}
As before $\gamma^{\square}\in\P^{\square}G$ and  the subgroup
$P^{\square}G$ coincides with the image of $\P G$ by the mapping
$\square:\gamma\mapsto \gamma^{\square}$.

\begin{lem}\label{zero}
For any $\gamma \in \P G$ we have $\gamma \sim \gamma^{\star}$ and
$\gamma\sim\gamma^{\square}$.
\end{lem}

\begin{proof}
We have to find a homotopy $\Gamma$ rel. endpoints between
$\gamma$ and $\gamma^{\star}$. For all $s\in I$ define $\Gamma$ as
follows:
\begin{equation}\nonumber
\Gamma(t,s)= \left\{
\begin{array}{lcl}
e& for & t\in [0,\frac{s}{2}]\\
\gamma(\frac{2t-s}{2-s})& for& t\in [\frac{s}{2},1]
\end{array}
\right.
\end{equation}
It is easy to check that such  $\Gamma$ fulfils all the
requirements. Analogously the second claim follows.
\end{proof}

After these prerequisites let us return to homeomorphism groups.

Let $X$ be a paracompact space and  let $G\leq \H(X)$. Here
$\H(X)$ is endowed with the compact-open topology and $G$ with the
induced topology.  If $f\in\P G$ then we define
$\supp(f):=\bigcup_{t\in[0,1]}\supp(f_t)$. By $G_0$ we define the
subgroup of all $g\in G$ such that there is $f\in\P G$ such that
$f_1=g$. $G_0$ is called the \emph{identity component} of $G$.
Clearly $G_0\lhd G$.
\begin{dff}
We say that $G$ is \emph{isotopically factorizable} if for every
open cover $\U$ and every isotopy $f\in\P G$ there are $U_1\ld
U_r\in\U$ and $f_1\ld f_r\in\P G$ such that $f=f_1\ldots f_r$ and
$\supp(f_i)\s U_i$ for all $i$.
\end{dff}

Clearly, if $G$ is isotopically factorizable then $G_0$ is
factorizable.

\medskip

\emph{Proof of Theorem 1.4} For $f\in\P G$ by $\langle
f\rangle_{\sim}$ denote the homotopy rel. endpoints class of $f$.

Due to Proposition 3.2 and the assumption, for any $x\in X$ there
is $g,\bar g\in[G_0,G_0]$ such that $[g,\bar g](x)\neq x$.
Consequently,  there exists  an open cover $\U$ such that for all
$U\in\U$ there are $g,\bar g\in[G_0,G_0]$ such that $[g,\bar
g](U)\cap U=\emptyset$. Since $G_0\lhd G$, the same holds for
$\U^G$ instead of $\U$. In view of Lemma 5.1, there are $f, \bar
f\in\P^{\square}G$ such that $f_1=g$ and $\bar f_1=\bar g$.

  Choose $\V$ such that $\V\prec\U$ (Lemma 3.3) and denote
\begin{equation*}\P G^{\U}=\prod\limits_{U\in \U^G}[\P G_U,\P G_U].\end{equation*}
 First we notice that
$[\P G,\P G]\s \P G^{\U}$. As in the proof of Theorem 1.3 we use
(3.1) for elements of $\P G$ and the fact that $\P G$ is
$\V$-factorizable.

 Next, fix arbitrarily $U\in \U$ and let $f, \bar f\in\P^{\square}G$ as above. Put $\hat f=[f,\bar f]$.
 Then $\hat f_t(U)\cap U=\emptyset$ for all $t\in [\frac{1}{2}, 1]$. We
will show that for every $h, \bar h\in \P G_U$ the bracket
$[\langle h\rangle_{\sim},\langle\bar h\rangle_{\sim}]$ is
represented as a product of four commutators of elements of
$[\tilde G,\tilde G]$. In view of Lemma 5.1 choose $k,\bar
k\in\P^{\star}G$ such that $\langle k\rangle_{\sim}=\langle
h\rangle_{\sim}$ and $\langle \bar k\rangle_{\sim}=\langle \bar
h\rangle_{\sim}$. It follows that $[\hat fk\hat f^{-1}, \bar
k]=\id$ and $[[\hat f,k],\bar k]=[k,\bar k]$. Therefore, $[\langle
h\rangle_{\sim},\langle\bar h\rangle_{\sim}]$  is a product of
four commutators of elements of $[\tilde G,\tilde G]$.  \quad
$\square$

\begin{rem}
(1) Observe that one can formulate some results for $[\tilde
G,\tilde G]$, analogous to Theorems  1.3, 3.5 and 3.8, by assuming
that $G$ is isotopically factorizable, $G_0$ satisfies some
conditions in Def. 3.1,  $\cl_{\P G}$ is bounded, and $\P G$ is
bounded in $\frag^{\U}$.

(2) Obviously, $\tilde G$ and $[\tilde G,\tilde G]$ are not
simple, since $\pi(G)\lhd\tilde G$ and $[\pi(G),\pi(G)]\lhd[\tilde
G,\tilde G]$, where $\pi(G)$ is the fundamental group of $G$.
\end{rem}

\section{The commutator subgroup of a diffeomorphism group on open
manifold}

Assume $r=0,1\ld\infty$. Let a manifold $M$ be the interior of a
compact, connected manifold $\bar M$ of class $C^r$ with non-empty
boundary $\p$. By a \wyr{product  neighborhood} of $\p$ we mean a
closed subset  $P=\p\t[0,1)$ of $M$ such that  $\p\t[0,1]$ is
embedded in $\bar M$, and $\p\t\{1\}$ is identified with $\p$.

A \wyr{translation system} on the product manifold $N\t[0,\infty)$
(c.f. \cite{li1}, p.168) is a family $\{P_j\}_{j=1}^{\infty}$ of
closed product neighborhoods of $N\t\{\infty\}$ such that
$P_{j+1}\s\intt P_j$ and $\bigcap_{j=1}^{\infty}P_j=\emptyset$. By
a {\it ball} we mean an open ball with its closure compact and
contained in a chart domain.

Let $G\leq\D^r(M)$, where $r=0,1\ld\infty$. For a subset $U\s M$
denote by $G(U)$ the subgroup of all elements of $G$ which can be
joined with the identity by an isotopy in $G$ compactly supported
in $U$.

\begin{dff}
Let $\B$ be a cover of $M$ by balls.
 $G$ is called
\wyr{$\B$-factorizable} if for any $f\in G$ there are  a product
neighborhood $P=\p\t[0,1)$, and  a family of diffeomorphisms
$g,g_1\ld g_{\rr}\in G$  such that:

(1) $f=g g_1\cdots g_{\rr}$ with $g\in G(P)$ and $g_j\in G(B_j)$,
where $B_j\in\B$ for $j=1\ld \rr$.

Furthermore, for any product neighborhood $P$ and for any $g\in
G(P)$ there is a sequence of reals from (0,1) tending to 1
\[0<a_1<\bar a_1<\bar b_1<b_1<a_2<\ldots<a_n<\bar a_n<\bar
b_n<b_n<\ldots<1\] and $h\in G(P)$ such that

(2) $h=g$ on $\bigcup_{n=1}^{\infty} \p\t[\bar a_n,\bar b_n]$;

(3) $h=\id$ if $g=\id$.

Put $D_n:=\p\t(a_n,b_n)$ and $D:=\bigcup_{n=1}^{\infty}D_n$. Then
we also assume that:

(4) $\supp(h)\s D$;

(5) for the resulting decomposition $h=h_1h_2\ldots$ with respect
to $D=\bigcup_{n=1}^{\infty}D_n$ we have $h_n\in G(D_n)$ for all
$n$.

 $G$ is called \wyr{factorizable (in the wider sense)} if it is
$\B$-factorizable for every cover $\B$ of $M$ by balls.

Finally, if $G$ factorizable, for any $f\in G$ we define
  $\Frag_{G}(f)$ as the smallest $\rr$ such that there are a family of balls
  $\{B_{j}\}$, a product neighborhood $P$ and
  and a decomposition of $f$ as in (1).  Then
   $\Frag_{G}$ is a conjugation-invariant norm on $G$,
called the \wyr{fragmentation norm}. In fact, since
$G\leq\D^r(M)$, any $g\in G$ does not change the ends of $M$ so
that it takes (by conjugation) any decomposition as in (1) into
another such a decomposition.

Define $\Fragd_{G}:=\sup_{g\in G}\Frag_{G}(g)$, the diameter of
$G$ in $\Frag_{G}$. Consequently, $\Frag_{G}$ is bounded iff
$\Fragd_{G}<\infty$.

\end{dff}

\begin{rem}
The reason for introducing Def. 6.1 is the absence of isotopy
extension theorems or fragmentation theorems for some geometric
structures. Roughly speaking, $G$ satisfies Def. 6.1 if all its
elements can be joined with id by an isotopy in $G$  and
appropriate versions of the above mentioned theorems are
available.
\end{rem}

Let $\diff^r(M)$ (resp. $\diff^r_c(M)$) be the group of all $C^r$
diffeomorphisms of $M$ (resp. with compact support). To illustrate
Def. 6.1 we consider the following
\begin{exa}
The group $\diff^r(\R^n)$ does not satisfy Def.6.1. The reason is
that in this case any $f\in\diff^r(\R^n)$ would be isotopic to id
due to 6.1(1) which is not true. Next, any $f\in\diff^r_c(\R^n)$
is isotopic  to the identity but the isotopy need not be compactly
supported. It follows that $\diff^r_c(\R^n)$ does not fulfil
Def.6.1.(1). The exception is $r=0$, when the Alexander trick is
in use (see e.g. \cite{ed-ki}, p.70) and any compactly supported
homeomorphism on $\R^n$ is isotopic to id by a compactly supported
isotopy. It follows that $\diff^0_c(\R^n)$ is factorizable in view
of \cite{ed-ki}.

Let $C=\R\t\mathbb S^1$ be the annulus. Then there is the twisting
number epimorphism $\diff^r_c(C)\r\mathbb Z$. It follows that
$\diff^r_c(C)$ is unbounded in view of Lemma 1.10 in \cite{bip}.
On the other hand, $\diff^r_c(C)$ is not factorizable.
\end{exa}

\begin{dff}
\begin{enumerate}
 \item $G$ is said to be \wyr{determined on compact subsets}
 if the following is satisfied. Let $f\in\D^r(M)$.
 If there are a sequence of relatively compact subsets  $U_1\s\overline U_1\s
 U_2\s\ldots\s U_n\s\overline{U}_n\s
 U_{n+1}\s\ldots$ with  $\bigcup U_n=M$ and  a sequence $\{g_n\}$, $n=1,2\ld$ of elements
 of $G$ such that $f|_{U_n}=g_n|_{U_n}$ for $n=1,2\ld$ then we
have $f\in G$.

\item We say that $G$ \wyr{admits translation systems} if  for any
sequence $\{\lambda_n\}$, $n=0,1,\ldots$, with
$\lambda_n\in(0,1)$, tending increasingly to 1, there exists a
$C^r$-mapping $[0,\infty)\ni t\mapsto f_t\in G$ supported in the
interior of $P$, with $f_0=\id$, $f_j=(f_1)^j$ for $j=2,3,\dots$,
and such that for the translation system $P_n=\p_i\t[\lambda_n,1)$
one has $f_1(P_n)=P_{n+1}$ for $n=0,1,2,\ldots$.
\end{enumerate}
\end{dff}

By using suitable isotopy extension theorems (c.f. \cite{ed-ki},
\cite{hir}, \cite{ban2}) we have

\begin{prop} \cite{ry6}
The groups $\D^r(M)$, $r=0,1\ld\infty$,  satisfy Definitions 6.1
and 6.4.
\end{prop}

The following result is essential to describe the structure of
$[G,G]$. Though it was proved in \cite{ry6}, we give the proof of
it for the sake of completeness.

\begin{lem}
If $G$ satisfies  Definitions 6.1 and 6.4, then any $g\in G(P)$,
where $P$ is a product neighborhood of $\p$, can be written as a
product of two commutators of elements of $G(P)$.
\end{lem}
\begin{proof}
We may assume that $g\in G(\intt( P))$. Choose as in Def. 6.1 a
sequence $0<a_1<\bar a_1<\bar b_1<b_1<a_2<\ldots<a_n<\bar a_n<\bar
b_n<b_n<\ldots<1$ and $h\in G(P)$ such that conditions (2)-(5) in
Def. 6.1 are fulfilled.
Put $\bar h=h^{-1}g$, that is $g=h\bar h$. Then $\supp(\bar h)$ is
in $(0,\bar a_1)\cup\bigcup_{n=1}^{\infty}(\bar b_n,\bar
a_{n+1})$, and $\bar h=g$ on  $[0, a_1]\cup\bigcup_{n=1}^{\infty}[
b_n, a_{n+1}]$. We show that $h$ is a commutator of elements in
$G(\intt(P))$.

Choose arbitrarily $\lambda_0\in (0,a_1)$ and
$\lambda_n\in(b_n,a_{n+1})$ for $n=1,2,\ldots$. In light of Def.
6.4(2) there exists an isotopy  $[0,\infty)\ni t\mapsto f_t\in G$
supported in $\p\t(0,1)$, such that $f_0=\id$ and
$f_j(P_n)=P_{n+j}$ for $j=1,2,\ldots$ and for $n=0,1,2,\ldots$,
where $P_n=\p\t[\lambda_n,1)$ for $n=0,1,\ldots$. Now define
$\tilde h\in G(\intt(P))$ as follows. Set $\tilde h=h$ on
$\p\t[0,\lambda_1)$, and $\tilde
h=h(f_1hf_1^{-1})\ldots(f_nhf_n^{-1})$ on $\p\t[0,\lambda_{n+1})$
for $n=1,2\ldots$. Here $f_n=(f_1)^n$. Then $\tilde
h|_{\p\t[0,\lambda_n)}$ is a consistent family of functions, and
$\tilde h=\bigcup_{n=1}^{\infty} \tilde h|_{\p\t[0,\lambda_n)}$ is
a local diffeomorphism. It is easily checked that $\tilde h$ is a
bijection. Due to Def. 6.4(1) $\tilde h\in G(\intt (P))$.

By definition we have the equality $\tilde h=hf_1\tilde h
f_1^{-1}$. It follows that $h=\tilde h f_1 \tilde
h^{-1}f_1^{-1}=[\tilde h,f_1]$. Similarly, $\bar h$ is a
commutator of elements of $G(P)$. The claim follows.
\end{proof}

\begin{dff}
Let $G$ satisfy Def. 6.1. Then
\begin{enumerate} \item the symbol $G_c$ stands for the subgroup
of all $f\in G$ such that there is a decomposition $f=gg_1\ldots
g_{\rr}$ as in Def. 6.1(1) with $g=\id$; \item $G$ is said to be
\emph{localizable} if for any $f\in G$ and any compact $C\s M$
there is $g\in G_c$ such that $f=g$ on $C$.
\end{enumerate}
\end{dff}
Clearly $G_c$ is a subgroup of the group of compactly supported
members of $G$. However, the converse is not true: for $G=\D^r(C)$
take a compactly supported diffeomorphism of $C$ with nonzero
twisting number (Example 6.3). For the reason of introducing
localizable groups, see Remark 6.2. It follows from the isotopy
extension theorems (\cite{ed-ki}, \cite{hir}) that $\D^r(M)$ is
localizable.

\begin{prop} Let $\frag_G=\frag_G^{\B}$, where $\B$ the family of all balls on $M$ (c.f. section 2). We have
$\fragd_{G_c}=\Fragd_{G_c}$.
\end{prop}
\begin{proof}
If $g\in G_c$ then $\Frag_{G_c}(g)\leq\frag_{G_c}(g)$, since any
fragmentation of $g$ supported in balls is of the form from Def.
6.1(1). On the other hand, if $g=g_0g_1\ldots g_{\rr'}$ with
$\rr'<\rr=\frag_{G_c}(g)$ is as in 6.1(1), then $g_0^{-1}g\in G_c$
and $\frag_{G_c}(g_0^{-1}g)\leq \rr'$. Thus,
$\fragd_{G_c}=\Fragd_{G_c}$.
\end{proof}

For any $M$ as above a theorem of McDuff \cite{md} states that
$\D^r(M)$ is perfect. We generalize it as follows.

\begin{thm} Let $M$ be an open $C^r$-manifold ($r=0,1\ld\infty$) such that $M=\intt\bar M$,
where $\bar M$ is a compact manifold. Suppose that  $G\leq\D^r(M)$
satisfies Definitions 6.1, 6.4 and 6.7, and that $G_c$ is
non-fixing. Then $[G,G]$ is perfect.
\end{thm}

\begin{proof} In view of Def. 6.1 for an arbitrary
$f\in G$ we can write $f=gh$, where $g\in G(P)$ and $h\in G_c$.
Let $[f_1,f_2]\in [G,G]$ with $f_1=g_1h_1$ and $f_2=g_2h_2$ as
above. Since $G_c$ is localizable we have $[g_1,h_2], [g_2,h_1]\in
[G_c,G_c]$. Due to Lemma 6.6 $G(P)$ is perfect, that is
$g_1,g_2\in[G,G]$. It follows from (3.1) that
$[f_1,f_2]=\phi[k_1,k_2][k'_1,k'_2][k''_1,k''_2]$, where
$\phi\in[[G,G],[G,G]]$ and $k_1,k_2,k'_1,k'_2,k''_1,k''_2\in G_c$.
But by Theorem 1.2 $[G_c,G_c]$ is also perfect. It follows  that
$[G,G]$ is perfect too. \end{proof}

\begin{thm}
Under the assumptions of Theorem 6.9, if $\cl_{G_c}$ and
$\frag_G^{\V}$ are bounded, where $\V$ is an arbitrary open cover
with $\V\prec\B$, then $[G,G]$ is uniformly perfect.
\end{thm}
\begin{proof} By Theorem 6.9, $[G,G]$ is perfect. In view of Proposition 3.2, $[G,G]$ is
1-non-fixing. Due to this fact and Lemma 3.3  we can find an open
cover $\U$ such that $\U\prec\B$ and such that for each $U\in\U$
there are $h_1,h_2\in[G,G]$ with $U\cap [h_1,h_2](U)=\emptyset$.
We denote
\begin{equation*}G^{\U}=\prod\limits_{U\in \U^G}[G_U,G_U].\end{equation*}
Here  $\U^G:=\{g(U):\; g\in G \text{\; and \;} U\in\U\}$. Then
also for each $U\in\U^G$ there is $h_1,h_2\in[G,G]$ with $U\cap
[h_1,h_2](U)=\emptyset$.

Assume   that  $\V\prec\U$  and $\fragd^{\V}_G= \rr$. Let
$[f_1,f_2]\in[G,G]$. As in the proof of Theorem 6.9 we have
$$[f_1,f_2]=[g_1,g_2][h_1,h_2][h'_1,h'_2][h''_1,h''_2],$$
where $g_1,g_2\in G(P)$ and $h_1\ld h''_2\in G_c$. By Lemma 6.6
and (3.1), $[g_1,g_2]$ is a product of four commutators of
elements of $[G,G]$.

Next, any $[h_1,h_2]\in[G_c,G_c]$ can be expressed as a product of
at most $\rr^2$  elements of $G^{\U}$ of the form $[k_1,k_2]$,
where $k_1,k_2\in G_U$ for some $U$. In fact, it is a consequence
of (3.1) and the fact that  $\V\prec\U$. Now if $\cld_{G_c}=d$,
then every element of $[G_c,G_c]$ is a product of at most $d\rr^2$
elements of $G^{\U}$ of the form $[k_1,k_2]$, where $k_1,k_2\in
G_U$ for some $U$.

 Finally, fix arbitrarily $U\in \U^G$. We wish to show that for
every $k_1,k_2\in G_U$ the bracket $[k_1,k_2]$ can be represented
as a product of four commutators of   elements of $[G,G]$.
By assumption on $\U^G$, there are $h_1,h_2\in [G,G]$  such that
$h(U)\cap U=\emptyset$ for $h=[h_1,h_2]$. It follows that
$[hk_1h^{-1}, k_2]=\id$. Therefore, $[[h,k_1],k_2]=[k_1,k_2]$.
Observe that indeed $[[h,k_1],k_2]$ is a product of four
commutators of elements of $[G,G]$. Thus any element of $[G,G]$ is
a product of at most $4d(1+\rr^2)$ commutators of elements of
$[G,G]$.

\end{proof}
\begin{cor}
Suppose that the assumptions of Theorem 6.9 are fulfilled and that
$G$ is bounded. Then $[G,G]$ is uniformly perfect.
\end{cor}
In fact, $\cl_G$ is bounded in view of Proposition 1.4 in
\cite{bip}, and $\frag_{G_c}$ is bounded in view of Proposition
6.8.

\begin{rem}
By using Theorems 3.5 and 3.8, Lemma 6.6 and (3.1) we can obtain
some estimates on $\cl_{[G,G]}$.
\end{rem}

\section{Examples and open problems}

Let $M$ be a paracompact manifold, possibly with boundary, of
class $C^{r}$, $r=0,1, \ldots,\infty$.

\medskip

{\bf 1.} Let $M$ be a manifold with a boundary,
$\mathrm{dim}(M)=n\geqslant
    2$. Then $G=\D^{r}_{c}(M)$, where $r=0,1,\dots
    ,\infty$, $r\neq n$ and $r\neq n+1$ is perfect ( \cite{ry2}, \cite{ry}) and non-simple.
Recently, Abe and Fukui \cite{af09}, using results of Tsuboi
\cite{Tsu2} and their own methods, showed that $G$ is also
uniformly perfect for many types of $M$. In the remaining cases,
where we do not know whether $G$ is perfect or uniformly perfect,
our results are of use.
    \medskip

{\bf 2.}
    Let $N$ be a  submanifold of $M$ of class $C^r$, {$r=0,1,\ldots ,\infty$}, and $\dim N\geq 1$.
    It was proved in \cite{ry3} that $G_c$, where $G=\D^r(M,N)$ is the identity component of the group
    of $C^r$-diffeomorphisms preserving $N$,  is perfect.
    The same was proved in the Lipschitz category in \cite{af}. All these groups are clearly non-simple.
    It follows from \cite{af09} that $G_c$ is also uniformly
    perfect for many types of pairs $(M,N)$. Several results of
    the present paper give new information on the structure of $G$
    and $G_c$.

\medskip
{\bf 3.} Given a foliation $\mathcal{F}$ of dimension $k$ on a
    manifold $M$, let $G=\D^{r}(M,\mathcal{F})$ be the identity
    component
    group of all diffeomorphisms of class $C^r$ taking each leaf to
    itself. Due to results of Rybicki \cite{ry1}, Fukui and Imanishi \cite{fi} and Tsuboi \cite{Tsu1},
    the group $G_c$ is
    perfect provided $r=0,1,\dots ,k$ or $r=\infty$.
     It is very
    likely that for large (but finite) $r$  the group $\D_c^{r}(M,\mathcal{F})$ is not
    perfect (c.f. a discussion on this problem in \cite{le-ry}). It is
    a highly non-trivial problem whether $G_c$ is uniformly
    perfect. Several results of the present paper apply to $G_c$ or
    $G$.

\medskip

{\bf 4.} Let  $\mathcal{F}$ be a foliation of dimension $k$ on the
Lipschitz
    manifold $M$ and let $G=\mathrm{Lip}(M,\mathcal{F})$ be the
    group of all Lipschitz homeomorphisms taking each leaf of $\F$ to
    itself. In view of results of Fukui and Imanishi \cite{fi1},
    the group $G_c$ is
    perfect. Further results may be concluded from our paper.

\medskip

{\bf 5.}
    Assume now that $\mathcal{F}$ is a singular foliation,
    i.e. the dimensions of its leaves need not be equal (see \cite{st}). One can
    consider the  group of~leaf-preserving
    diffeomorphisms of $\mathcal{F}$,
    $G=\D^{\infty}(M,\mathcal{F})$. However, it is
    hopeless to obtain any perfectness results for this group. On the other
    hand,
    Theorem 1.2 still works in this case and we know that the
    commutator group $[G_c,G_c]$ is perfect. We do not know whether
    $[G_c,G_c]$ is uniformly perfect.
\medskip

{\bf 6.}
    Let us recall the definition of Jacobi manifold (see \cite{dlm}).
    Let $M$ be a $C^{\infty}$ manifold, let $\frak{X}(M)$ be the
    Lie algebra of the vector fields on $M$ and denote by
    $C^{\infty}(M,\mathbb{R})$ the algebra of $C^{\infty}$
    real-valued functions on $M$. A \emph{Jacobi structure} on $M$
    is a pair $(\Lambda, E)$, where $\Lambda$ is a 2-vector field and
    $E$ is a vector field on $M$ satisfying
    $$[\Lambda, \Lambda]=2E \wedge \Lambda,\quad [E,\Lambda]=0.$$
    Here, $[\, ,\,]$ is the Schouten-Nijenhuis bracket. The
    manifold $M$ endowed with the Jacobi structure is called a \emph{Jacobi
    manifold}. If $E=0$ then
    $(M,\Lambda)$ is a Poisson manifold.
    Observe that the notion of Jacobi manifold generalizes also symplectic, locally conformal
    symplectic and contact manifolds.

    Now, let $(M,\Lambda,E)$ be a Jacobi manifold.
    A diffeomorphism $f$ on $M$ is called a \emph{hamiltonian diffeomorphism}
    if, by definition, there exists a hamiltonian isotopy $f_t$, $t\in
    [0,1]$, such that $f_0=\id$ and $f_1=f$. An isotopy $f_t$ is
    \emph{hamiltonian} if the corresponding time-dependent vector
    field $X_t=\dot{f}_t\circ f_{t}^{-1}$ is hamiltonian.

    Let $G=\mathcal{H}(M,\Lambda,E)$ be
    the compactly supported identity component of all hamiltonian
    diffeomorphisms of class $C^{\infty}$ of $(M,\Lambda,E)$. It
    is not known whether $G$ is perfect,
    even in the case of regular Poisson manifold (\cite{ry4}). However, by
    Theorem 1.2 the commutator group $[G,G]$ is
    perfect. It is an interesting and difficult problem to answer when $[G,G]$ is uniformly perfect.

     In the transitive cases, the compactly supported
    identity components of the hamiltonian symplectomorphism group
    and the contactomorphism group are simple (\cite{ban1}, \cite{ha-ry}, \cite{ry5}).
In general, $G$ and $\tilde G$ is not uniformly perfect in the
symplectic case, see \cite{bip}. An obstacle for the uniform
simplicity of the first group is condition (2) in Theorem 4.3. On
the other hand, the contactomorphism group satisfies this
condition and it is likely that for some contact manifolds it is
uniformly simple.

\end{document}